\RequirePackage{fix-cm}
\documentclass[smallextended]{svjour3}       
\smartqed  
\usepackage{graphicx}
\usepackage{eurosym}
\usepackage{makeidx}
\usepackage{amsfonts}
\usepackage{amsmath}
\usepackage{graphicx}
\usepackage{amssymb}
\usepackage{mathrsfs}
\usepackage[all,cmtip]{xy}
\usepackage{synttree}
\usepackage{tikz}
\usepackage{color}
\usepackage{xcolor}
\usepackage{hyperref}
\newcommand{\BIGOP}[1]{\mathop{\mathchoice{\raise-0.22em\hbox{\huge $#1$}} {\raise-0.05em\hbox{\Large $#1$}}{\hbox{\large $#1$}}{#1}}}
\journalname{}

\newcommand{\level}{\mathrm{level}}
\newcommand{\depth}{\mathrm{depth}}

\begin{document}

\title{Tree-based tensor formats
}


\author{Antonio Falc\'o   \and
        Wolfgang Hackbusch \and 
        Anthony Nouy
}


\institute{A. Falc\'o \at
              ESI International Chair@CEU-UCH, \\
Departamento de Matem\'aticas, F\'{\i}sica y Ciencias
Tecnol\'ogicas,\\
Universidad Cardenal Herrera-CEU, CEU Universities \\
San Bartolom\'e 55,
46115 Alfara del Patriarca (Valencia), Spain
              \email{afalco@uchceu.es}           
           \and
           W. Hackbusch \at
              Max-Planck-Institut \emph{Mathematik in den Naturwissenschaften}\\
Inselstr. 22, D-04103 Leipzig, Germany 
\email{wh@mis.mpg.de}
            \and
            A. Nouy \at
            Centrale Nantes, \\
LMJL UMR CNRS 6629\\
1 rue de la No\"e,
44321 Nantes Cedex 3, France.
\email{anthony.nouy@ec-nantes.fr}
}

\date{Received: date / Accepted: date}

\maketitle

\begin{abstract}
The main goal of this paper is to study the topological properties of tensors in tree-based Tucker format. These formats include the Tucker format and the Hierarchical Tucker format. 
A property of the so-called minimal subspaces is used for obtaining a representation of tensors with either bounded or fixed tree-based rank in the underlying algebraic tensor space.  We provide a new characterisation of minimal subspaces which extends the existing characterisations.   
We also introduce a definition of topological tensor
spaces in tree-based format, with the introduction of a norm at each vertex of the tree,  and prove the existence of best
approximations from sets of tensors with bounded tree-based rank,  under some assumptions on the norms weaker than in the existing results. 
\end{abstract}

\section{Introduction}

Tensor approximation methods play a central role in the numerical solution of high dimensional problems arising in a wide range of applications. The reader is referred to the monograph \cite{Hackbusch} and surveys \cite{Khoromskij:2012fk,Grasedyck:2013,nouy:2017_morbook,bachmayr2016tensor} for an introduction to tensor numerical methods and an overview of recent developments in the field.
Low-rank tensor formats based on subspaces are widely used for complexity reduction in the representation of high-order tensors. 
Two of the most popular formats are the
Tucker format and the Hierarchical Tucker format \cite{HaKuehn2009} (HT for
short). It is possible to show that the Tensor Train format \cite{Osedelets1}, introduced originally by Vidal \cite{Vidal}, is a particular
case of the HT format (see e.g. Chapter 12 in \cite{Hackbusch}). In the framework of topological tensor
spaces, first results have been obtained on the existence of a best approximation in each fixed set of
tensors with bounded rank \cite{FALHACK}. In particular, this allows to
construct, on a theoretical level, iterative methods for
nonlinear convex {optimisation} problems over reflexive tensor Banach spaces \cite%
{FalcoNouy}. More generally, this is a crucial property for proving the stability of algorithms using tree-based tensor formats.

The Tucker and the HT formats are completely
characterised by a rooted tree together with a finite sequence of natural
numbers associated to each vertex of the tree, denominated the tree-based
rank. Each number in the tree-based rank is associated with a class of
subspaces of fixed dimension. It can be shown that for a given
tree, every element in the tensor space possesses a unique  tree-based rank.
In consequence, given a tree, a tensor space is a union of sets indexed by
the tree-based ranks. It allows to consider for a given tree two kinds of
sets in a tensor space: the set of tensors of fixed tree-based rank and the
set of tensors of bounded tree-based rank. 

This paper provides new results on the representation of tensors in general tree-based Tucker formats, in particular on a characterisation of minimal subspaces compatible with a given tree. It also provides a definition of topological tensor spaces associated with a given tree, and provides new results on the existence of best approximations from sets of tensors with bounded tree-based rank.

The paper is organised as follows. In Section~\ref{Tensor_TBF} , we introduce the tree-based tensors as a
generalisation, at algebraic level, of the hierarchical tensor format.  Moreover, we provide a new characterisation of
the minimal subspaces of tree-based tensors extending the previous results
obtained in \cite{FALHACK}, and introduce the definition of tree-based
rank.  Another main result of this section is Theorem~\ref{characterisation_FT}, which provides a characterisation for 
the representation for the set of tensors with
fixed tree-based rank.
 In Section~\ref{TTT_TBF} we introduce a definition of topological tensor spaces in 
tree-based format, with the introduction of a norm at each vertex of the tree. Finally in  Section~\ref{TTT_TBF},
we prove the existence of best approximations from sets 
of tensors with bounded tree-based rank under some assumptions on the norms that are weaker than the ones introduced in \cite{FALHACK}. 

\section{Algebraic tensors in the tree-based format}
\label{Tensor_TBF}

\subsection{Preliminary definitions and notations}

Let $D=\{1,2,\ldots,d\}$ be a finite index set, and let  $V_{j}$ $\left( 1\leq j\leq
d\right) $, be vector spaces.
Concerning the definition of the algebraic tensor space 
\begin{equation*}
\mathbf{V}_D:=\left. _{a}\bigotimes_{j=1}^{d}V_{j}\right.,
\end{equation*}
we refer to Greub \cite{Greub}. As underlying field we choose $%
\mathbb{R},$ but the results hold also for $\mathbb{C}$. The suffix `$a$' in
$_{a}\otimes_{j=1}^{d}V_{j}$ refers to the `algebraic' nature. By
definition, all elements of $\mathbf{V}$ 
are \emph{finite} linear combinations of elementary tensors $\mathbf{v}%
=\otimes_{j=1}^{d}v_{j}$ $\left( v_{j}\in V_{j}\right).$

For vector spaces $V_{j}$ and $W_{j}$ over $\mathbb{R},$ let linear mappings
$A_{j}:V_{j}\rightarrow W_{j}$ $\left( 1\leq j\leq d\right) $ be given. Then
the definition of the elementary tensor%
\begin{equation*}
\mathbf{A}=\bigotimes_{j=1}^{d}A_{j}:\;\mathbf{V}_D=\left. _{a}\bigotimes
_{j=1}^{d}V_{j}\right. \longrightarrow\mathbf{W}_D=\left. _{a}\bigotimes
_{j=1}^{d}W_{j}\right.
\end{equation*}
is given by%
\begin{equation}
\mathbf{A}\left( \bigotimes_{j=1}^{d}v_{j}\right)
:=\bigotimes_{j=1}^{d}\left( A_{j}v_{j}\right) .
\label{(A als Tensorprodukt der Aj}
\end{equation}
Note that \eqref{(A als Tensorprodukt der Aj} uniquely defines the linear
mapping $\mathbf{A}:\mathbf{V}_D\rightarrow\mathbf{W}_D.$ We recall that $L(V,W)$
is the space of linear maps from $V$ into $W,$ while $V^{\prime}=L(V,\mathbb{%
R})$ is the algebraic dual of $V$. For normed spaces, $\mathcal{L}(V,W)$
denotes the continuous linear maps, while $V^{\ast}=\mathcal{L}(V,\mathbb{R}%
) $ is the topological dual of $V$.

\subsection{Minimal subspaces in tensor representations}

 For a given  $\alpha \in 2^D \setminus \{\emptyset,D\}$, we let $\mathbf{V}_{\alpha}:= \left._a \bigotimes_{j \in \alpha} V_{j} \right.,$ with the convention $\mathbf{V}_{\{j\}} = V_j$ for all $j\in D$. The algebraic tensor space $\mathbf{V}_D$ is identified with $\mathbf{V}_\alpha \otimes_a \mathbf{V}_{\alpha^c}$, where $\alpha^c = D \setminus \alpha$. For a tensor $\mathbf{v}\in  \mathbf{V}_D =\mathbf{V}_\alpha \otimes_a \mathbf{V}_{\alpha^c}$ ,
the \emph{minimal subspace} $U_{\alpha}^{\min }(\mathbf{v}) \subset \mathbf{V}_\alpha$ of $\mathbf{v}$ is defined by the properties that $\mathbf{v} \in U_{\alpha}^{\min }(\mathbf{v}) \otimes_a \mathbf{V}_\alpha$ and $\mathbf{v} \in \mathbf{U}_{\alpha} \otimes_a \mathbf{V}_\alpha$ implies $U_{\alpha}^{\min }(\mathbf{v}) \subset \mathbf{U}_{\alpha}  $. Here we use the notation $U_{\{j\}}^{\min}(\mathbf{v}) = U_{j}^{\min}(\mathbf{v}) $, and we adopt the convention $U_{D}^{\min}(\mathbf{v}) = \mathrm{span}{\{\mathbf{v}\}}$.
We recall some useful results on minimal subspaces (see Section 2.2 in \cite{FALHACK}).
\begin{proposition}\label{minimal}
Let $\mathbf{v}\in \mathbf{V}_D$. For any $\alpha \in 2^D \setminus \{\emptyset,D\}$, there exists a unique minimal subspace $U_{\alpha}^{\min}(\mathbf{v}),$ where  $\dim
U_{\alpha}^{\min}(\mathbf{v}) < \infty$. Furthermore, it holds $\dim U_{\alpha}^{\min}(\mathbf{v}) = \dim U_{\alpha^c}^{\min}(\mathbf{v})$.
\end{proposition}
The relation between minimal subspaces is as follows (see Corollary~2.9 of \cite%
{FALHACK}).
\begin{proposition}\label{inclusin_Umin} 
Let $\mathbf{v}\in \mathbf{V}_D$. For any $\alpha \in 2^D$ with $\#\alpha \ge 2$ and a non-trivial partition $\mathcal{P}_{\alpha}$ of $\alpha$, it holds 
 \begin{equation*}
U_{\alpha}^{\min }(\mathbf{v})\subset \left.
_{a}\bigotimes_{\beta \in \mathcal{P}_{\alpha}}U_{\beta}^{\min }(\mathbf{v})\right. .
\end{equation*}
\end{proposition}

Let $\mathcal{P}_D$ be a given non-trivial partition of $D.$ The algebraic tensor space $ \mathbf{V}_D = \left._a \bigotimes_{j=1}^d V_{j} \right.$ is identified with $ \left._a \bigotimes_{\alpha \in \mathcal{P}_D} \mathbf{V}_{\alpha} \right..$ By definition of the minimal subspaces $ U_{\alpha}^{\min}(\mathbf{v}),$ $\alpha \in \mathcal{P}_D$, we have 
$$
\mathbf{v} \in \left._a \bigotimes_{\alpha \in \mathcal{P}_D} U_{\alpha}^{\min}(\mathbf{v}%
)\right..$$
For a given $\alpha\in \mathcal{P}_D$ with $\#\alpha \ge 2$ and a non-trivial partition $\mathcal{P}_{\alpha}$ of $\alpha$, we also have 
\begin{equation*}
 \mathbf{v} \in \left(\left._a \bigotimes_{\beta \in \mathcal{P}_{\alpha}}
U_{\beta}^{\min}(\mathbf{v}) \right. \right) \otimes_ a \left(\left._a
\bigotimes_{\delta \in \mathcal{P}_D\setminus \{\alpha\}} U_{\delta}^{\min}(\mathbf{v}) \right. \right).
\end{equation*}
The following result gives a characterisation  of minimal subspaces. 
\begin{proposition}
\label{(Ualpha in Tensor Uj coro} Let $\mathbf{v} \in \mathbf{V}_D$ and let $\alpha$ be a subset of $D$ with $\#\alpha\ge 2$ and 
$\mathcal{P}_\alpha$ be  a non-trivial partition of $\alpha$. Assume that $\mathbf{V}_{\alpha}$ and $\mathbf{V}_{\beta}$, for $\beta \in \mathcal{P}_\alpha
$, are normed spaces. Then for each $\beta \in \mathcal{P}_\alpha
$, it holds
\begin{align*}
U_{\beta}^{\min}(\mathbf{v}) & =  \mathrm{span}\, \left\{ \left( id_{\beta} \otimes\boldsymbol{\varphi}%
^{(\alpha \setminus \beta)} \right)(\mathbf{v}_{\alpha}): \mathbf{v}%
_{\alpha} \in U_{\alpha }^{\min }(\mathbf{v}) \text{, } \boldsymbol{%
\varphi}^{(\alpha \setminus \beta)} \in \left._{a}\bigotimes_{\gamma \in \mathcal{P}_\alpha \setminus \{\beta\}} 
\mathbf{V}_{\beta}^* \right. \right\}
\end{align*}
\end{proposition}

\begin{proof}
First observe that 
$
\mathbf{V}_{D}=\mathbf{V}_{\alpha }\otimes _{a}\mathbf{V}_{\alpha^c}=\left(
\left. _{a}\bigotimes_{\beta \in \mathcal{P}_\alpha} \mathbf{V}_{\beta}\right. \right)
\otimes _{a} \mathbf{V}_{\alpha^c} .
$
From  \cite[Theorem 2.17]{FALHACK}, we have   
$
U_{\alpha }^{\min }(\mathbf{v}) = \left\{ (id_{\alpha }\otimes \boldsymbol{%
\varphi }^{(\alpha^c )})(\mathbf{v}):\boldsymbol{\varphi }^{(\alpha^c )}\in \mathbf{V}_{\alpha^c}^\ast
 \right\}.
$ 
Since $\mathbf{v} \in \mathbf{V}_\alpha \otimes U_{\alpha^c }^{\min }(\mathbf{v})$, we can replace $\mathbf{V}_{\alpha^c}^\ast$ by 
the larger space $U_{\alpha^c }^{\min }(\mathbf{v})^\ast$, and obtain
\begin{align*}
U_{\alpha }^{\min }(\mathbf{v})& =\left\{ (id_{\alpha }\otimes \boldsymbol{%
\varphi }^{(\alpha^c )})(\mathbf{v}):\boldsymbol{\varphi }^{(\alpha^c )}\in
U_{\alpha^c }^{\min }(\mathbf{v})^\ast\right\}.
\end{align*}
In a similar way and again from  \cite[Theorem 2.17]{FALHACK}, we also prove that for any $\beta \in \mathcal{P}_\alpha$, it holds
\begin{align*}
U_{\beta}^{\min }(\mathbf{v})& =\left\{ (id_{\beta}\otimes 
\boldsymbol{\varphi }^{(\beta^c)})(\mathbf{v}):\boldsymbol{%
\varphi }^{(\beta^c)}\in \left( \left. _{a}\bigotimes_{\gamma \in \mathcal{P}_\alpha \setminus \{\beta\}
} U_{\gamma}^{\min }(\mathbf{v})^{\ast }\right. \right) \otimes_a  U_{\alpha^c}^{\min }(\mathbf{v}%
)^{\ast } \right\}.
\end{align*}%
Take $\mathbf{v}_{\alpha }\in {U}_{\alpha }^{\min }(%
\mathbf{v}).$ Then there exists $\boldsymbol{\varphi }^{(\alpha^c )}\in U_{\alpha^c }^{\min }(\mathbf{v})^\ast $
such that $\mathbf{v}_{\alpha }=\left( id_{\alpha }\otimes \boldsymbol{%
\varphi }^{(\gamma )}\right) (\mathbf{v}).$ Now, for $\boldsymbol{\varphi }%
^{(\alpha \setminus \beta)}\in \left. _{a}\bigotimes_{\gamma \in \mathcal{P}_\alpha \setminus \{\beta\}}U_{\gamma}^{\min }(\mathbf{v})^{\ast }\right. ,$ we have 
\begin{equation*}
\left( id_{\beta}\otimes \boldsymbol{\varphi }^{(\alpha \setminus
\beta)}\right) (\mathbf{v}_{\alpha })=\left( id_{\beta}\otimes 
\boldsymbol{\varphi }^{(\alpha \setminus \beta)}\otimes \boldsymbol{%
\varphi }^{(\alpha^c )}\right) (\mathbf{v}),
\end{equation*}%
and hence $\left( id_{\beta}\otimes \boldsymbol{\varphi }^{(\alpha
\setminus \beta)}\right) (\mathbf{v}_{\alpha })\in U_{\beta}^{\min }(\mathbf{v}).$ This proves a first inclusion. Now for $\beta\in \mathcal{P}_\alpha$, take $\mathbf{v}_{\beta }\in U_{\beta}^{\min }(\mathbf{v}),$ then there exists 
\begin{equation*}
\boldsymbol{%
\varphi }^{(\beta^c)}\in \left( \left. _{a}\bigotimes_{\gamma \in \mathcal{P}_\alpha \setminus \{\beta\}
} U_{\gamma}^{\min }(\mathbf{v})^{\ast }\right. \right) \otimes_a  U_{\alpha^c}^{\min }(\mathbf{v}%
)^{\ast }
\end{equation*}%
such that $\mathbf{v}_{\beta}=\left( id_{\beta}\otimes 
\boldsymbol{\varphi }^{(\beta^c)}\right) (\mathbf{v}).$ Then $%
\boldsymbol{\varphi }^{(\beta^c)}=\sum_{l=1}^{r}\boldsymbol{%
\psi }_{l}^{(\alpha \setminus \beta)}\otimes \boldsymbol{\phi }%
_{l}^{(\alpha^c )},$ where $\boldsymbol{\phi }_{l}^{(\alpha^c )}\in  U_{\alpha^c}^{\min }(\mathbf{v})^{\ast } $
and $\boldsymbol{\psi }_{l}^{(\alpha \setminus \beta)}\in \left. _{a}\bigotimes_{\gamma \in \mathcal{P}_\alpha \setminus \{\beta\}
} U_{\gamma}^{\min }(\mathbf{v})^{\ast }\right.,$
for $1\leq l\leq r.$ Thus, 
\begin{align*}
\mathbf{v}_{\beta}& =\left( id_{\beta}\otimes \boldsymbol{%
\varphi }^{(\beta^c)}\right) (\mathbf{v})   =\sum_{i=1}^{r}\left( id_{\beta}\otimes \boldsymbol{\psi }%
_{i}^{(\alpha \setminus \beta)}\otimes \boldsymbol{\phi }_{i}^{(\alpha^c
)}\right) (\mathbf{v}) \\
& =\sum_{i=1}^{r}\left( id_{\beta}\otimes \boldsymbol{\psi }%
_{i}^{(\alpha \setminus \beta)}\right) \left( (id_{\alpha }\otimes 
\boldsymbol{\phi }_{i}^{(\alpha^c )})(\mathbf{v})\right) .
\end{align*}%
Observing that $(id_{\alpha }\otimes \boldsymbol{\phi }_{l}^{(\alpha^c )})(%
\mathbf{v})\in U_{\alpha }^{\min }(\mathbf{v}),$ we obtain the other inclusion.
\qed
\end{proof}

\subsection{Algebraic tensor spaces in the tree-based format}

\begin{definition}
\label{partition_tree} A tree $T_{D}$ is called \emph{a dimension
partition tree of $D$} if
\begin{enumerate}
\item[(a)] all vertices $\alpha \in T_D$ are non-empty subsets of $D,$
\item[(b)] $D$ is the \emph{root} of $T_D,$
\item[(c)] every vertex $\alpha \in T_{D}$ with $\#\alpha \geq 2$ has at
least two sons and the set of sons of $\alpha$, denoted $S(\alpha)$, is a non-trivial partition of $\alpha$,  
\item[(d)] every vertex $\alpha\in T_D$ with $\#\alpha = 1$ has no son and is called a  \emph{leaf}.
\end{enumerate}
\end{definition}

The set of leaves is denoted by $\mathcal{L}(T_{D}).$ A straightforward consequence of Definition~%
\ref{partition_tree} is that the set of leaves $\mathcal{L}(T_{D})$
coincides with the singletons of $D,$ i.e., $\mathcal{L}(T_{D})=\{\{j\}:j\in
D\} $ and hence it is the trivial partition of $D.$
We remark that for a tree $T_D$ such that $S(D)\neq \mathcal{L}(T_D)$, $S(D)$ is a non-trivial partition of $D.$

We denote by $\mathrm{level}(\alpha)$, $\alpha \in T_D$, the levels of the vertices in $T_D$, which are such that $\mathrm{level}(D) = 0$ and for any pair $\alpha,\beta \in T_D$ such that $\beta \in S(\alpha),$ $\mathrm{level}(\beta) = \mathrm{level}(\alpha)+1$. The \emph{depth}\footnote{By using the notion of edge, that is, the connection between one vertex to another, then our definition of depth coincides with the classical definition of height, i.e. the longest downward path between the root and a leaf.} of the tree $T_D$ is defined 
as $\mathrm{depth}(T_D) = \max_{\alpha \in T_D} \level(\alpha)  .$

\begin{definition}
For a tensor space $\mathbf{V}_D$ and a dimension 
partition tree $T_{D}$, the pair $(\mathbf{V}_D,T_D)$ is called a representation of the tensor space
$\mathbf{V}_{D}$ in \emph{tree-based format}, and is associated with the collection of spaces $\{\mathbf{V}_\alpha\}_{\alpha \in T_D \setminus D}$.
\end{definition}

\begin{example}[Tucker format]\label{example-tucker}
In Figure~\ref{fig2p}, $D=\{1,2,3,4,5,6\} $ and $$T_D=\{D,\{1\},\{2\},\{3\},\{4\},\{5\},\{6\}\}.$$
Here $\depth(T_D)=1$. This corresponds to the Tucker format.
\begin{figure}[h]
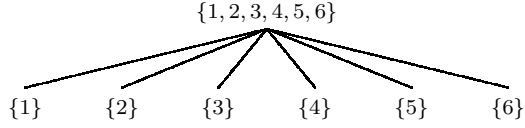

\centering
\synttree[$\{1,2,3,4,5,6\}$[$\{1\}$][$\{2\}$][$\{3\}$]
[$\{4\}$][$\{5\}$][$\{6\}$]]
\caption{Tuker format: dimension partition tree $\depth(T_D)=1$ with $S(D) =   \mathcal{L}(T_D).$}
\label{fig2p}
\end{figure}\end{example}

\begin{example}
\label{example_nobinary} In Figure~\ref{fig1p}, $D=\{1,2,3,4,5,6\}$ and 
\begin{equation*}
T_D=\{D,\{1,2,3\},\{4,5\},\{1\},\{2\},\{3\},\{4\},\{5\},\{6\}\}.
\end{equation*}
Here $\depth(T_D)=2$.
\begin{figure}[h]
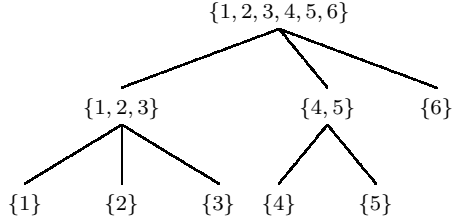

\centering
\synttree[$\{1,2,3,4,5,6\}$[$\{1,2,3\}$[$\{1\}$][$\{2\}$][$\{3\}$]]
[$\{4,5\}$[$\{4\}$][$\{5\}$]][$\{6\}$ ]]
\caption{A dimension partition tree with $\depth(T_D)=2$.
}\label{fig1p}
\end{figure}

\end{example}
%
%

Let $\mathbb{N}_{0}:=\mathbb{N} \cup\{0\} $ denote the set of non-negative integers.
For each $\mathbf{v%
} \in \mathbf{V}_D$, we have that $(\dim {U}_{\alpha}^{\min}(\mathbf{v}%
))_{\alpha \in 2^D \setminus \{\emptyset\}}$
is in $\mathbb{N}_0^{2^{\#D}-1}.$ 

\begin{definition}
For a given partition dimension tree $T_D$ over $D$, and for each $\mathbf{v}\in
\mathbf{V}_{D}$, we define its
\emph{tree-based rank} by the tuple
$\mathrm{rank}_{T_D}( \mathbf{v} ):= (\dim {U}_{\alpha }^{\min }(%
\mathbf{v}))_{\alpha \in T_{D}}\in \mathbb{N}_0^{\#T_{D}}.$
\end{definition}

\begin{definition}
We will say that $\mathfrak{r}:=(r_{\alpha })_{\alpha \in T_{D}}\in \mathbb{N%
}^{\#T_{D}}$ is an \emph{admissible tuple for $T_{D}$,} if there exists $%
\mathbf{v}\in \mathbf{V}_{D}$ such that $\dim
U_{\alpha }^{\min }(\mathbf{v})=r_{\alpha }$ for all $\alpha \in
T_{D}.$ We will denote the set of admissible ranks for the representation
$(\mathbf{V}_D,T_D)$ of the tensor space $\mathbf{V}_D$ by
$$
\mathcal{AD}(\mathbf{V}_D,T_D):=\{(\dim U_{\alpha }^{\min }(\mathbf{v}))_{\alpha \in T_D}:\mathbf{v} \in \mathbf{V}_D\}.
$$
\end{definition}

\subsection{The set of tensors in tree-based format with fixed or bounded tree-based rank}
\label{Sec_Hierar}

\begin{definition}
Let $T_{D}$ be a given dimension partition tree and fix some tuple $%
\mathfrak{r}\in \mathcal{AD}(\mathbf{V}_D,T_D)$. Then \emph{the set of 
tensors of fixed tree-based rank $\mathfrak{r}$} is defined by
\begin{equation}
\mathcal{FT}_{\mathfrak{r}}(\mathbf{V}_{D},T_D):=\left\{ \mathbf{v}\in  \mathbf{V}_{D}:\dim {U}_{\alpha }^{\min }(\mathbf{v}%
)=r_{\alpha }\text{ for all }\alpha \in T_{D}\right\}
\end{equation}%
and the \emph{set of tensors of tree-based rank bounded by $\mathfrak{r}$} is
defined by
\begin{equation}
\mathcal{FT}_{\leq \mathfrak{r}}(\mathbf{V}_{D},T_D):=\left\{ \mathbf{v}\in
\mathbf{V}_{D}:%
\begin{array}{l}
\dim {U}_{\alpha }^{\min }(\mathbf{v})\leq r_{\alpha }\text{ for all }\alpha
\in T_{D}%
\end{array}%
\right\} .  \label{(Hr}
\end{equation}
\end{definition}

For $\mathfrak{r},\mathfrak{s}\in \mathbb{%
N}_0^{\#T_{D}}$ we write $\mathfrak{s}\leq \mathfrak{r}$ if and only if $%
s_{\alpha }\leq r_{\alpha }$ for all $\alpha \in T_{D}.$ Then fo a fixed
$\mathfrak{r} \in  \mathcal{AD}(\mathbf{V}_D,T_D)$, we have 
\begin{equation}  \label{connected_id}
\mathcal{FT}_{\le \mathfrak{r}}(\mathbf{V}_{D},T_D) :=
\bigcup_{\substack{\mathfrak{s}\leq \mathfrak{r} \\ \mathfrak{s}
\in \mathcal{AD}(\mathbf{V}_D,T_D)}}\mathcal{FT}_{\mathfrak{s}}(\mathbf{V}%
_{D},T_D).
\end{equation}%
We point out that in \cite{FHN} is introduced a representation of $\mathbf{V}_D$ in Tucker format. 
Letting $T_D^{\text{Tucker}}$ be the Tucker dimension partition tree (see example \ref{example-tucker})
and given $\mathfrak{r} \in \mathcal{AD}(\mathbf{V}_D,T_D^{\text{Tucker}})$, we define the set of tensors with fixed Tucker rank $\mathfrak{r}$ by
$$
\mathfrak{M}_{\mathfrak{r}}(\mathbf{V}_D):= \mathcal{FT}_{\mathfrak{r}}(\mathbf{V}_{D},T_D^{\text{Tucker}}) =
\left\{\mathbf{v} \in \mathbf{V}_D: \dim U_{k}^{\min}(\mathbf{v}) = r_{k}, \, k \in \mathcal{L}(T_D^{\text{Tucker}}) \right\}.
$$
Then
$$
\mathbf{V}_D = \bigcup_{\mathfrak{r} \in  \mathcal{AD}(\mathbf{V}_D,T_D^{\text{Tucker}})}\mathfrak{M}_{\mathfrak{r}}(\mathbf{V}_D).
$$

\subsection{The representation of tensors in tree based format with fixed tree based rank}

Before stating the next result we recall the definition of the `matricisation' (or `unfolding') of a tensor in a
finite-dimensional setting.

\begin{definition}
\label{Def Malpha} 
Let $\alpha $ be a finite set of indices, $\mathcal{P}_{\alpha}$ be a non-trivial partition of $\alpha$, and $r= (r_\mu)_{\mu \in \mathcal{P}_\alpha} \in \mathbb{N}^{\#\mathcal{P}_\alpha}$. 
 For $\beta \in \mathcal{P}_{\alpha}$, we define a
map $\mathcal{M}_{\beta}$ 
\begin{equation*}
\begin{tabular}{llll}
$\mathcal{M}_{\beta}:$ & $\mathbb{R}^{%
\mathop{\mathchoice{\raise-0.22em\hbox{\huge $\times$}} {\raise-0.05em\hbox{\Large $\times$}}{\hbox{\large
$\times$}}{\times}}_{\mu \in \mathcal{P}_{\alpha}}r_{\mu}} $ & $\rightarrow $ & $\mathbb{R}%
^{r_{\beta  } \times \left( \prod_{\mu \in
\mathcal{P}_{\alpha} \setminus \{\beta\}}r_{\mu}\right) },$ \\
& $C_{(i_{\mu})_{\mu \in \mathcal{P}_{\alpha}}}$ & $\mapsto $ & $C_{i_{\beta},(i_{\mu})_{\mu \in \mathcal{P}_{\alpha} \setminus \{\beta\}}}$
\end{tabular},
\end{equation*}
which is an isomorphism.  
Given $C \in \mathbb{R}^{%
\mathop{\mathchoice{\raise-0.22em\hbox{\huge $\times$}}
{\raise-0.05em\hbox{\Large $\times$}}{\hbox{\large
$\times$}}{\times}}_{\mu \in \mathcal{P}_{\alpha}}r_{\mu}}$
we have that $C \in\mathfrak{M}_r\left(
\mathbb{R}^{%
\mathop{\mathchoice{\raise-0.22em\hbox{\huge $\times$}}
{\raise-0.05em\hbox{\Large $\times$}}{\hbox{\large
$\times$}}{\times}}_{\mu \in \mathcal{P}_{\alpha}}r_{\mu}}\right)$ if and only if $\mathrm{rank}\, \mathcal{M}_{\beta}(C) = r_{\beta}$ 
 for each $%
\beta \in \mathcal{P}_{\alpha},$ or equivalently $%
\mathcal{M}_{\beta}(C)\mathcal{M}_{\beta}(C)^T \in \mathrm{GL}(\mathbb{R}^{r_{\beta}})$ for $\beta \in \mathcal{P}_{\alpha}.$ 
\end{definition}

The next result gives us a characterisation of the tensors in $\mathcal{FT}_{%
\mathfrak{r}}(\mathbf{V}_{D},T_D)$.

\begin{theorem}
\label{characterisation_FT} Let $%
T_{D}$ be a dimension partition tree over $D$ with $\depth(T_D) = \mathfrak{d}.$ 
Given $\mathfrak{r} \in \mathcal{AD}(\mathbf{V}_D,T_D)$ then the 
following statements are equivalent.

\begin{enumerate}
\item[(a)] $\mathbf{v}\in \mathcal{FT}_{\mathfrak{r}}(\mathbf{V}_{D},T_D).$

\item[(b)] Given $\{u_{i_k}^{(k)}:1 \le i_k \le r_k\}$ a fixed basis of $%
U_k^{\min}(\mathbf{v})$ for $k\in \mathcal{L}(T_D)$,  
\begin{equation}
\mathbf{v}=\sum_{\substack{ 1\leq i_{\alpha }\leq r_{\alpha }  \\ \alpha \in
S(D)}}C_{(i_{\alpha })_{\alpha \in S(D)}}^{(D)}\bigotimes_{\alpha \in S(D)}%
\mathbf{u}_{i_{\alpha }}^{(\alpha )},\label{TBRT1}
\end{equation}%
for a unique $C^{(D)} \in \textcolor{red}{\mathfrak{M}_{r}(\mathbb{R}^{ 
\mathop{\mathchoice{\raise-0.22em\hbox{\huge $\times$}}
{\raise-0.05em\hbox{\Large $\times$}}{\hbox{\large
$\times$}}{\times}}_{\beta \in S(D)}r_{\beta }})}$ and
where for each $\mu \in
T_D \setminus \{D\}$ such that $S(\mu )\neq \emptyset ,$ there exists a unique ${C}^{(\mu)}
\in \mathbb{R}^{r_{\mu} \times 
\mathop{\mathchoice{\raise-0.22em\hbox{\huge $\times$}}
{\raise-0.05em\hbox{\Large $\times$}}{\hbox{\large
$\times$}}{\times}}_{\beta \in S(\mu)}r_{\beta }}$ such that 
$\mathrm{rank}\, \mathcal{M}_{\mu}(C^{(\mu)})=\dim U_{\mu}^{\min}(\mathbf{v}) 
= r_{\mu},$ and the set $\{%
\mathbf{u}_{i_{\mu }}^{(\mu )}:1\leq i_{\mu }\leq r_{\mu}\},$ with
\begin{equation}
\mathbf{u}_{i_{\mu }}^{(\mu )}=\sum_{\substack{ 1\leq i_{\beta }\leq
r_{\beta }  \\ \beta \in S(\mu )}}C_{i_{\mu },(i_{\beta })_{\beta \in S(\mu
)}}^{(\mu )}\bigotimes_{\beta \in S(\mu )}\mathbf{u}_{i_{\beta }}^{(\beta )}
\label{TBRT2}
\end{equation}%
for $1\leq i_{\mu }\leq r_{\mu },$ is a basis of $U_{\mu}^{\min }(\mathbf{v}%
).$
\end{enumerate}
\end{theorem}

\begin{proof}
(b) clearly implies (a). Now consider $\mathbf{v} \in \mathcal{FT}_{\mathfrak{r}}(\mathbf{V}_{D},T_D).$
Since $\mathbf{v} \in \bigotimes_{\alpha \in S(D)} U^{\min}_{\alpha}(\mathbf{v}) $,
there exists a unique $C^{(D)} \in  \mathbb{R}^{ 
\mathop{\mathchoice{\raise-0.22em\hbox{\huge $\times$}}
{\raise-0.05em\hbox{\Large $\times$}}{\hbox{\large
$\times$}}{\times}}_{\beta \in S(D)}r_{\beta }}$ such that 
$$
\mathbf{v}=\sum_{\substack{ 1\leq i_{\alpha }\leq r_{\alpha }  \\ \alpha \in
S(D)}}C_{(i_{\alpha })_{\alpha \in S(D)}}^{(D)}\bigotimes_{\alpha \in S(D)}%
\mathbf{u}_{i_{\alpha }}^{(\alpha )},
$$
where $\{\mathbf{u}_{i_{\alpha}}^{(\alpha)}: 1 \le i_{\alpha} \le r_{\alpha}\}$ is a fixed basis
of $U_{\alpha}^{\min}(\mathbf{v})$ for $\alpha \in S(D).$  
\textcolor{red}{Since $\mathrm{rank}\,\mathcal{M}_{\alpha}(C^{(D)}) = \dim U_{\alpha}^{\min}(\mathbf{v}) = r_{\alpha}$
for each $\alpha \in S(D),$ we have that $C^{(D)} \in \mathfrak{M}_{r}(\mathbb{R}^{ 
\mathop{\mathchoice{\raise-0.22em\hbox{\huge $\times$}}
{\raise-0.05em\hbox{\Large $\times$}}{\hbox{\large
$\times$}}{\times}}_{\beta \in S(D)}r_{\beta }}).$}
Now, for each $\mu \in
T_{D}\setminus \{D\}$ such that $S(\mu )\neq \emptyset ,$ thanks to
Proposition~\ref{inclusin_Umin}, we have 
\begin{equation*}
U_{\mu }^{\min }(\mathbf{v})\subset \left. _{a}\bigotimes_{\beta \in S(\mu
)}U_{\beta }^{\min }(\mathbf{v})\right..
\end{equation*}
Consider $\{\mathbf{u}_{i_{\mu }}^{(\mu )}:1\leq i_{\mu }\leq r_{\mu }\}$ a
basis of $U_{\mu }^{\min }(\mathbf{v})$ and $\{\mathbf{u}_{i_{\beta
}}^{(\beta )}:1\leq i_{\beta }\leq r_{\beta }\}$ a basis of $U_{\beta
}^{\min }(\mathbf{v})$ for $\beta \in S(\mu )$ and $1\leq i_{\mu }\leq
r_{\mu }.$ Then, there exists a unique $C^{(\mu )}\in \mathbb{R}^{r_{\mu
}\times \left( 
\mathop{\mathchoice{\raise-0.22em\hbox{\huge $\times$}}
{\raise-0.05em\hbox{\Large $\times$}}{\hbox{\large
$\times$}}{\times}}_{\beta \in S(\alpha )}r_{\beta }\right) }$ such that 
\begin{equation*}
\mathbf{u}_{i_{\mu }}^{(\mu )}=\sum_{\substack{ 1\leq i_{\beta }\leq
r_{\beta }  \\ \beta \in S(\mu )}}C_{i_{\mu },(i_{\beta })_{\beta \in S(\mu
)}}^{(\mu )}\bigotimes_{\beta \in S(\mu )}\mathbf{u}_{i_{\beta }}^{(\beta )},
\end{equation*}%
for $1\leq i_{\mu }\leq r_{\mu }.$ Since $\{\mathbf{u}_{i_{\mu }}^{(\mu
)}:1\leq i_{\mu }\leq r_{\mu }\}$ is a basis, then
\begin{equation}\label{rank_relation}
\mathrm{rank}\,\mathcal{M}_{\mu }(C^{(\mu )}) = \dim U_{\mu}^{\min}(\mathbf{v}) = r_{\mu}
,\end{equation}
holds for each $\mu \in
T_{D}\setminus \{D\}$ such that $S(\mu )\neq \emptyset.$ Then (c) holds.
 \qed
\end{proof}

\section{Topological tensor spaces in the tree-based format}\label{TTT_TBF}

First, we recall the definition of tensor Banach spaces.

\begin{definition}
\label{Banach tensor product space}We say that $\mathbf{V}_{\left\Vert
\cdot\right\Vert }$ is a \emph{Banach tensor space} if there exists an
algebraic tensor space $\mathbf{V}$ and a norm $\left\Vert \cdot\right\Vert $
on $\mathbf{V}$ such that $\mathbf{V}_{\left\Vert \cdot\right\Vert }$ is the
completion of $\mathbf{V}$ with respect to the norm $\left\Vert
\cdot\right\Vert $, i.e.
\begin{equation*}
\mathbf{V}_{\left\Vert \cdot\right\Vert }:=\left. _{\left\Vert \cdot
\right\Vert }\bigotimes_{j=1}^{d}V_{j}\right. =\overline{\left.
_{a}\bigotimes\nolimits_{j=1}^{d}V_{j}\right. }^{\left\Vert \cdot\right\Vert
}.
\end{equation*}
If $\mathbf{V}_{\left\Vert \cdot\right\Vert }$ is a Hilbert space, we say
that $\mathbf{V}_{\left\Vert \cdot\right\Vert }$ is a \emph{Hilbert tensor
space}.
\end{definition}

Next, we give some examples of Banach and Hilbert tensor spaces.

\begin{example}
\label{Bsp HNp}For $I_{j}\subset \mathbb{R}$ $\left( 1\leq j\leq d\right) $
and $1\leq p<\infty ,$ the Sobolev space $H^{N,p}(I_{j})$ consists of all
univariate functions $f$ from $L^{p}(I_{j})$ with bounded norm\footnote{%
It suffices to have in \eqref{(SobolevNormp a} the terms $n=0$ and $n=N.$
The derivatives are to be understood as weak derivatives.}%
\begin{subequations}
\label{(SobolevNormp}%
\begin{equation*}
\left\Vert f\right\Vert _{N,p;I_{j}}:=\bigg(\sum_{n=0}^{N}\int_{I_{j}}\left%
\vert \partial ^{n}f\right\vert ^{p}\mathrm{d}x\bigg)^{1/p},
\label{(SobolevNormp a}
\end{equation*}%
whereas the space $H^{N,p}(\mathbf{I})$ of $d$-variate functions on $\mathbf{%
I}=I_{1}\times I_{2}\times \ldots \times I_{d}\subset \mathbb{R}^{d}$ is
endowed with the norm%
\begin{equation*}
\left\Vert f\right\Vert _{N,p}:=\Big(\sum_{0\leq \left\vert \mathbf{n}%
\right\vert \leq N}\int_{\mathbf{I}}\left\vert \partial ^{\mathbf{n}%
}f\right\vert ^{p}\mathrm{d}\mathbf{x}\Big)^{1/p}  \label{(SobolevNormp b}
\end{equation*}%
\end{subequations}%
with $\mathbf{n}\in \mathbb{N}_{0}^{d}$ being a multi-index of length $%
\left\vert \mathbf{n}\right\vert :=\sum_{j=1}^{d}n_{j}$. For $p>1$ it is
well known that $H^{N,p}(I_{j})$ and $H^{N,p}(\mathbf{I})$ are reflexive and
separable Banach spaces. Moreover, for $p=2,$ the Sobolev spaces $%
H^{N}(I_{j}):=H^{N,2}(I_{j})$ and $H^{N}(\mathbf{I}):=H^{N,2}(\mathbf{I})$
are Hilbert spaces. As a first example,%
\begin{equation*}
H^{N,p}(\mathbf{I})=\left. _{\left\Vert \cdot \right\Vert
_{N,p}}\bigotimes_{j=1}^{d}H^{N,p}(I_{j})\right.
\end{equation*}%
is a Banach tensor space. Examples of Hilbert tensor spaces are%
\begin{equation*}
L^{2}(\mathbf{I})=\left. _{\left\Vert \cdot \right\Vert
_{0,2}}\bigotimes_{j=1}^{d}L^{2}(I_{j})\right. \text{\quad and\quad }H^{N}(%
\mathbf{I})=\left. _{\left\Vert \cdot \right\Vert
_{N,2}}\bigotimes_{j=1}^{d}H^{N}(I_{j})\right. \text{ for }N\in \mathbb{N}.
\end{equation*}
\end{example}

In the definition of a tensor Banach space $\left._{\|\cdot\|} \bigotimes_{j
\in D} V_j \right.$ we have not fixed whether  the $V_j,$ for $j \in D,$ are
complete or not. This leads us to introduce the following definition.

\begin{definition}
Let $D$ be a finite index set and $T_{D}$ be a dimension partition tree over $D$. Let 
$(V_{j},\Vert \cdot \Vert _{j})$ be a normed space such that $V_{j_{\Vert
\cdot \Vert _{j}}}$ is a Banach space obtained by the completion of $V_{j},$
for $j\in D,$ and consider a representation $\{\mathbf{V}_{\alpha
}\}_{\alpha \in T_{D}\setminus \{D\}}$ of the tensor space $\mathbf{V}%
_{D}=\left. _{a}\bigotimes_{j\in D}V_{j}\right. $ where for each $\alpha \in
T_{D}\setminus \mathcal{L}(T_{D})$ we have a tensor space $\mathbf{V}%
_{\alpha }=\left. _{a}\bigotimes_{\beta \in S(\alpha )}\mathbf{V}_{\beta
}\right. .$ If for each $\alpha \in T_{D}\setminus \mathcal{L}(T_{D})$ there
exists a norm $\Vert \cdot \Vert _{\alpha }$ defined on $\mathbf{V}_{\alpha
} $ such that $\mathbf{V}_{\alpha _{\Vert \cdot \Vert _{\alpha }}}=\left.
_{\Vert \cdot \Vert _{\alpha }}\bigotimes_{\beta \in S(\alpha )} \mathbf{V}_{\beta
}\right. $ is a tensor Banach space, we say that $\{\mathbf{V}_{\alpha
_{\Vert \cdot \Vert _{\alpha }}}\}_{\alpha \in T_{D}\setminus \{D\}}$ is a
representation of the tensor Banach space $\mathbf{V}_{D_{\Vert \cdot \Vert
_{D}}}=\left. _{\Vert \cdot \Vert _{D}}\bigotimes_{j\in D}V_{j}\right. $ in
the \emph{topological tree-based format}.
\end{definition}

For $\alpha\in T_D \setminus \mathcal{L}(T_D)$,
\begin{equation*}
\mathbf{V}_{\alpha_{\|\cdot\|_{\alpha}}} = 
\left._{\|\cdot\|_{\alpha}} \bigotimes_{j\in \alpha} V_{j}\right. = \left._{\|\cdot\|_{\alpha}} \bigotimes_{\beta \in S(\alpha)} \mathbf{V}_{\beta}\right. .
\end{equation*}
\begin{example}
Figure \ref{figZp} gives an example of a representation in the topological
tree-based format for an anisotropic Sobolev space.
\end{example}

\begin{figure}[h]
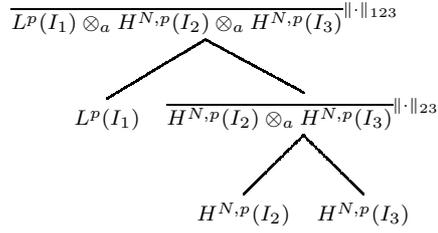

\centering
\synttree[$\overline{L^p(I_1) \otimes _a H^{N,p}(I_2) \otimes_{a} H^{N,p}(I_3)}^{\|\cdot\|_{123}}$
[$L^p(I_1)$]
[$\overline{H^{N,p}(I_2) \otimes_{a} H^{N,p}(I_3)}^{\|\cdot \|_{23}}$[$H^{N,p}(I_2)$][$H^{N,p}(I_3)$]
]]
\caption{A representation in the topological tree-based format for the
tensor Banach space $\overline{L^p(I_1) \otimes _a H^{N,p}(I_2) \otimes_{a}
H^{N,p}(I_3)}^{\|\cdot\|_{123}}.$ Here $\|\cdot\|_{23}$ and $\|\cdot\|_{123}$
are given norms.}
\label{figZp}
\end{figure}

\begin{remark}
Observe that the example in Figure \ref{figZp1} is not included in the
definition of the topological tree-based format. Moreover, for a tensor $%
\mathbf{v}\in L^{p}(I_{1})\otimes _{a}(H^{N,p}(I_{2})\otimes _{\Vert \cdot
\Vert _{23}}H^{N,p}(I_{3})),$ we have $U_{23}^{\min }(\mathbf{v})\subset
H^{N,p}(I_{2})\otimes _{\Vert \cdot \Vert _{23}}H^{N,p}(I_{3}).$ However, in
the topological tree-based representation of Figure \ref{figZp}, for a given 
$\mathbf{v}\in L^{p}(I_{1})\otimes _{a}H^{N,p}(I_{2})\otimes
_{a}H^{N,p}(I_{3})$ we have $U_{23}^{\min }(\mathbf{v})\subset
H^{N,p}(I_{2})\otimes _{a}H^{N,p}(I_{3}),$ and hence $U_{23}^{\min }(\mathbf{%
v})\subset U_{2}^{\min }(\mathbf{v})\otimes _{a}U_{3}^{\min }(\mathbf{v}).$
\end{remark}
\begin{figure}[h]
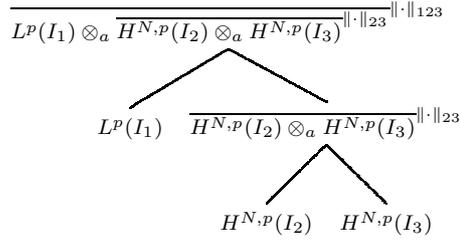

\centering
\synttree[$\overline{L^p(I_1) \otimes _a \overline{H^{N,p}(I_2) \otimes_a H^{N,p}(I_3)}^{\|\cdot \|_{23}}}^{\|\cdot\|_{123}}$
[$L^p(I_1)$]
[$\overline{H^{N,p}(I_2) \otimes_a H^{N,p}(I_3)}^{\|\cdot \|_{23}}$[$H^{N,p}(I_2)$][$H^{N,p}(I_3)$]
]]
\caption{A representation for the tensor Banach space $\overline{%
L^{p}(I_{1})\otimes _{a}\overline{H^{N,p}(I_{2})\otimes _{a}H^{N,p}(I_{3})}%
^{\Vert \cdot \Vert _{23}}}^{\Vert \cdot \Vert _{123}},$ using a tree. Here $%
\Vert \cdot \Vert _{23}$ and $\Vert \cdot \Vert _{123}$ are given norms.}
\label{figZp1}
\end{figure}

The difference between the tensor spaces involved in Figure \ref{figZp} and
Figure \ref{figZp1} is given by the fact that since
\begin{equation*}
H^{N,p}(I_2) \otimes_{a} H^{N,p}(I_3) \subset
\overline{H^{N,p}(I_2) \otimes_a H^{N,p}(I_3)}^{\|\cdot \|_{23}}
\end{equation*}
then
\begin{equation*}
\overline{L^p(I_1) \otimes _a H^{N,p}(I_2) \otimes_{a} H^{N,p}(I_3)}^{\|\cdot\|_{123}} \subset
\overline{L^p(I_1) \otimes _a \overline{H^{N,p}(I_2) \otimes_a H^{N,p}(I_3)}}^{\|\cdot \|_{23}}.
\end{equation*}
A desirable property for the tensor product is that if $\|\cdot\|_{\alpha}$
for each $\alpha \in T_D \setminus \mathcal{L}(T_D)$ is a norm on the 
tensor space $\left._{a} \bigotimes_{\beta \in
S(\alpha)} \mathbf{V}_{\beta_{\|\cdot\|_{\beta}}}\right.,$ then 
\begin{align}  \label{desirable_property}
\left._{\|\cdot\|_{\alpha}} \bigotimes_{\beta \in S(\alpha)} \mathbf{V}%
_{\beta_{\|\cdot\|_{\beta}}}\right. = \left._{\|\cdot\|_{\alpha}}
\bigotimes_{\beta \in S(\alpha)} \mathbf{V}_{\beta}\right. =
\left._{\|\cdot\|_{\alpha}} \bigotimes_{j\in \alpha} V_{j}\right.
\end{align}
must be true. To precise
these ideas, we introduce the following definitions and results.

Let $\left\Vert \cdot\right\Vert _{j},$ $1\leq j\leq d,$ be the norms of the
vector spaces $V_{j}$ appearing in $\mathbf{V}_D=\left. _{a}\bigotimes
\nolimits_{j=1}^{d}V_{j}\right. .$ By $\left\Vert \cdot\right\Vert_D $ we
denote the norm on the tensor space $\mathbf{V}_D$. Note that $%
\left\Vert\cdot\right\Vert_D$ is not determined by $\left\Vert
\cdot\right\Vert _{j},$ for $j \in D,$ but there are relations which are
`reasonable'. Any norm $\left\Vert \cdot\right\Vert $ on $\left.
_{a}\bigotimes_{j=1}^{d}V_{j}\right. $ satisfying%
\begin{equation}
\Big{\|}%
\bigotimes\nolimits_{j=1}^{d}v_j%
\Big{\|}%
=\prod\nolimits_{j=1}^{d}\Vert v_j\Vert_{j}\qquad\text{for all }v_j\in V_{j}%
\text{ }\left( 1\leq j\leq d\right)  \label{(rcn a}
\end{equation}
is called a \emph{crossnorm}. As usual, the dual norm of $\left\Vert
\cdot\right\Vert $ is denoted by $\left\Vert \cdot\right\Vert ^{\ast}$. If $%
\left\Vert \cdot\right\Vert $ is a crossnorm and also $\left\Vert
\cdot\right\Vert ^{\ast}$ is a crossnorm on $\left.
_{a}\bigotimes_{j=1}^{d}V_{j}^{\ast}\right. $, i.e.,%
\begin{equation}
\Big{\|}%
\bigotimes\nolimits_{j=1}^{d}\varphi^{(j)}%
\Big{\|}%
^{\ast}=\prod\nolimits_{j=1}^{d}\Vert\varphi^{(j)}\Vert_{j}^{\ast}\qquad%
\text{for all }\varphi^{(j)}\in V_{j}^{\ast}\text{ }\left( 1\leq j\leq
d\right) ,  \label{(rcn b}
\end{equation}
then $\left\Vert \cdot\right\Vert $ is called a \emph{reasonable crossnorm}.

\begin{remark}
\label{tensor product continuity}Eq. \eqref{(rcn a} implies the inequality $%
\Vert\bigotimes\nolimits_{j=1}^{d}v_j\Vert\lesssim\prod\nolimits_{j=1}^{d}%
\Vert v_j\Vert_{j}$ which is equivalent to the continuity of the multilinear
tensor product mapping\footnote{%
Recall that a multilinear map $T$ from $%
\mathop{\mathchoice{\raise-0.22em\hbox{\huge $\times$}} {\raise-0.05em\hbox{\Large $\times$}}{\hbox{\large
$\times$}}{\times}}_{j=1}^{d} (V_{j},\|\cdot\|_j)$ equipped with the product
topology to a normed space $(W,\Vert\cdot\Vert)$ is continuous if and only
if $\Vert T\Vert <\infty$, with 
\begin{align*}
\| T \| := \sup_{\substack{ (v_1,\hdots,v_d)  \\ \|(v_1,\hdots,v_d)\|\le 1}}
\|T(v_{1},\ldots,v_{d})\| &= \sup_{\substack{ (v_1,\hdots,v_d)  \\ %
\|v_1\|_1\le 1,\hdots, \|v_d\|_d\le 1}} \|T(v_{1},\ldots,v_{d})\| = \sup_{{%
(v_1,\hdots,v_d)}} \frac{\|T (v_{1},\ldots,v_{d})\|}{\|v_1\|_1\hdots %
\|v_d\|_d}.
\end{align*}%
} between normed spaces: 
\begin{equation}
\bigotimes :\mathop{\mathchoice{\raise-0.22em\hbox{\huge $\times$}}
{\raise-0.05em\hbox{\Large $\times$}}{\hbox{\large $\times$}}{\times}}%
_{j=1}^{d}\left( V_{j},\left\Vert \cdot\right\Vert _{j}\right)
\longrightarrow%
\bigg(%
\left. _{a}\bigotimes_{j=1}^{d}V_{j}\right. ,\left\Vert \cdot\right\Vert 
\bigg)%
,  \label{(Tensorproduktabbildung}
\end{equation}
defined by $\bigotimes\left( (v_{1},\ldots,v_{d})\right)
=\bigotimes_{j=1}^{d}v_{j} $, the product space being equipped with the
product topology induced by the maximum norm $\|(v_1,\hdots,v_d)\| =
\max_{1\le j\le d} \|v_j\|_j$.\smallskip
\end{remark}

The following result is a consequence of Lemma~4.34 of \cite{Hackbusch}.

\begin{lemma}
\label{lemma434} Let $(V_j,\|\cdot\|_j)$ be normed spaces for $1 \le j \le
d. $ Assume that $\|\cdot\|$ is a norm on the tensor space $\left.
_{a}\bigotimes_{j=1}^{d}V_{j_{\|\cdot\|_j}}\right.$ such that the tensor
product map 
\begin{equation}
\bigotimes :\mathop{\mathchoice{\raise-0.22em\hbox{\huge $\times$}}
{\raise-0.05em\hbox{\Large $\times$}}{\hbox{\large $\times$}}{\times}}
_{j=1}^{d}\left( V_{j_{\|\cdot\|_j}},\left\Vert \cdot\right\Vert _{j}\right)
\longrightarrow 
\bigg(
\left. _{a}\bigotimes_{j=1}^{d}V_{j_{\|\cdot\|_j}}\right. ,\left\Vert
\cdot\right\Vert 
\bigg)
\label{(Tensorproduktabbildung1}
\end{equation}
is continuous. Then \eqref{(Tensorproduktabbildung} is also continuous and 
\begin{align*}
\left. _{\|\cdot\|}\bigotimes_{j=1}^d V_{j_{\|\cdot\|_j}}\right. = \left.
_{\|\cdot\|}\bigotimes_{j=1}^d V_{j}\right.
\end{align*}
holds.
\end{lemma}

\begin{definition}
Assume that for each $\alpha \in T_D \setminus \mathcal{L}(T_D)$ there
exists a norm $\|\cdot\|_{\alpha}$ defined on $\left._{a} \bigotimes_{\beta
\in S(\alpha)} V_{\beta_{\|\cdot\|_{\beta}}} \right..$ We will say that the
tensor product map $\bigotimes$ is $T_D$-continuous if the map 
\begin{equation*}
\bigotimes: 
\mathop{\mathchoice{\raise-0.22em\hbox{\huge
$\times$}} {\raise-0.05em\hbox{\Large $\times$}}{\hbox{\large
$\times$}}{\times}}_{\beta \in
S(\alpha)}(V_{\beta_{\|\cdot\|_{\beta}}},\|\cdot\|_{\beta}) \rightarrow
\left(\left._a \bigotimes_{\beta \in S(\alpha)}
V_{\beta_{\|\cdot\|_{\beta}}} \right., \|\cdot\|_{\alpha} \right)
\end{equation*}
is continuous for each $\alpha \in T_D \setminus \mathcal{L}(T_D).$
\end{definition}

The next result gives the conditions to have \eqref{desirable_property}.

\begin{theorem}
\label{ext_Banach} Assume that we have a representation $\{\mathbf{V}%
_{\alpha _{\Vert \cdot \Vert _{\alpha }}}\}_{\alpha \in T_{D}\setminus
\{D\}} $ in the topological tree-based format of the tensor Banach space $%
\mathbf{V}_{D_{\Vert \cdot \Vert _{D}}}=\left. _{\Vert \cdot \Vert
_{D}}\bigotimes_{\alpha \in S(D)}\mathbf{V}_{\alpha }\right. ,$ such that
for each $\alpha \in T_{D}\setminus \mathcal{L}(T_{D})$, the norm $\Vert
\cdot \Vert _{\alpha }$ is also defined on $\left. _{a}\bigotimes_{\beta \in
S(\alpha )}V_{\beta _{\Vert \cdot \Vert _{\beta }}}\right. $ and the tensor
product map $\bigotimes $ is $T_{D}$-continuous. Then 
\begin{equation*}
\left. _{\Vert \cdot \Vert _{\alpha }}\bigotimes_{\beta \in S(\alpha )}%
\mathbf{V}_{\beta _{\Vert \cdot \Vert _{\beta }}}\right. =\left. _{\Vert
\cdot \Vert _{\alpha }}\bigotimes_{\beta \in S(\alpha )}\mathbf{V}_{\beta
}\right. =\left. _{\Vert \cdot \Vert _{\alpha }}\bigotimes_{j\in \alpha
}V_{j}\right. ,
\end{equation*}%
holds for all $\alpha \in T_{D}\setminus \mathcal{L}(T_{D}).$
\end{theorem}

\begin{proof}
From Lemma \ref{lemma434}, if the tensor product map 
\begin{equation*}
\bigotimes: 
\mathop{\mathchoice{\raise-0.22em\hbox{\huge $\times$}} {\raise-0.05em\hbox{\Large $\times$}}{\hbox{\large
$\times$}}{\times}}_{\beta \in S(\alpha)} (\mathbf{V}_{\beta_{\|\cdot\|_{%
\beta}}}, \|\cdot\|_{\beta}) \longrightarrow (\left._{a} \bigotimes_{\beta
\in S(\alpha)} \mathbf{V}_{\beta_{\|\cdot\|_{\beta}}}\right.,\|\cdot\|_{%
\alpha})
\end{equation*}
is continuous, then 
\begin{equation*}
\left._{\|\cdot\|_{\alpha}} \bigotimes_{\beta \in S(\alpha)} \mathbf{V}%
_{\beta_{\|\cdot\|_{\beta}}}\right.= \left._{\|\cdot\|_{\alpha}}
\bigotimes_{\beta \in S(\alpha)} V_{\beta}\right.,
\end{equation*}
holds. Since $\mathbf{V}_{\alpha} = \left._{a} \bigotimes_{\beta \in
S(\alpha)} \mathbf{V}_{\beta}\right. = \left._{a} \bigotimes_{j \in \alpha}
V_{j}\right.,$ the theorem follows. \qed
\end{proof}

\begin{example}
Assume that the tensor product maps 
\begin{equation*}
\bigotimes:(L^p(I_1),\|\cdot\|_{0,p;I_1}) \times (H^{N,p}(I_2)
\otimes_{\|\cdot\|_{23}} H^{N,p}(I_3), \|\cdot\|_{23}) \rightarrow (L^p(I_1)
\otimes_a (H^{N,p}(I_2) \otimes_{\|\cdot\|_{23}} H^{N,p}(I_3)),
\|\cdot\|_{123} )
\end{equation*}
and 
\begin{equation*}
\bigotimes:(H^{N,p}(I_2), \|\cdot\|_{N,p;I_2}) \times (H^{N,p}(I_3),
\|\cdot\|_{N,p;I_3}) \rightarrow (H^{N,p}(I_2) \otimes_{a} H^{N,p}(I_3),
\|\cdot\|_{23} )
\end{equation*}
are continuous. Then the trees of Figure \ref{figZp} and Figure \ref{figZp1}
are the same.
\end{example}

\section{On the best 
approximation in $\mathcal{FT}_{\le \mathfrak{r}}(\mathbf{V}_D)$}

Now we discuss about the best approximation problem 
in  $\mathcal{FT}_{\le \mathfrak{r}}(\mathbf{V}_D)$.
For this, we 
need a stronger condition than the $T_{D}$-continuity of the tensor product. Grothendieck 
\cite{Grothendiek1953} named the norm $\left\Vert \cdot
\right\Vert _{\vee }$ introduced below the \emph{injective norm}.

\begin{definition}
Let $V_{i}$ be a Banach space with norm $\left\Vert \cdot\right\Vert _{i}$
for $1\leq i\leq d.$ Then for $\mathbf{v}\in\mathbf{V}=\left.
_{a}\bigotimes_{j=1}^{d}V_{j}\right. $ define the norm $\left\Vert \cdot\right\Vert
_{\vee(V_1,\ldots,V_d)}$, called the injective norm, by%
\begin{equation}
\left\Vert \mathbf{v}\right\Vert _{\vee(V_1,\ldots,V_d)}:=\sup\left\{ \frac{%
\left\vert \left(
\varphi_{1}\otimes\varphi_{2}\otimes\ldots\otimes\varphi_{d}\right) (\mathbf{%
v})\right\vert }{\prod_{j=1}^{d}\Vert\varphi_{j}\Vert_{j}^{\ast}}%
:0\neq\varphi_{j}\in V_{j}^{\ast},1\leq j\leq d\right\} .
\label{(Norm ind*(V1,...,Vd)}
\end{equation}
\end{definition}

It is well known that the injective norm is a reasonable crossnorm (see
Lemma 1.6 in \cite{Light} and \eqref{(rcn a}-\eqref{(rcn b}). Further
properties are given by the next proposition (see Lemma 4.96 and 4.2.4 in 
\cite{Hackbusch}).

\begin{proposition}
\label{injective bounded below} Let $V_{i}$ be a Banach space with norm $%
\left\Vert \cdot\right\Vert _{i}$ for $1\leq i\leq d,$ and $\|\cdot\|$ be a
norm on $\mathbf{V}:= \left._a \bigotimes_{j=1}^d V_j\right..$ The following
statements hold.

\begin{itemize}
\item[(a)] For each $1 \le j \le d$ introduce the tensor Banach space $$%
\mathbf{X}_j:=\left._{\|\cdot\|_{\vee(V_1,\ldots,V_{j-1},V_{j+1},%
\ldots,V_d)}} \bigotimes_{k \neq j}V_k \right..$$ Then 
\begin{equation}
\|\cdot\|_{\vee(V_1,\ldots,V_d)} = \|\cdot\|_{\vee\left(V_j, \mathbf{X}_j
\right)}
\end{equation}
holds for $1 \le j \le d.$

\item[(b)] The injective norm is the weakest reasonable crossnorm on $%
\mathbf{V},$ i.e., if $\left\Vert \cdot\right\Vert $ is a reasonable
crossnorm on $\mathbf{V},$ then 
\begin{equation}
\left\Vert \cdot\right\Vert \;\gtrsim\ \left\Vert \cdot\right\Vert
_{\vee(V_1,\ldots,V_d)}.  \label{(Norm staerker als ind*}
\end{equation}

\item[(c)] For any norm $\left\Vert \cdot\right\Vert $ on $\mathbf{V}$
satisfying $\left\Vert \cdot\right\Vert
_{\vee(V_1,\ldots,V_d)}\lesssim\left\Vert \cdot\right\Vert ,$ the map (\ref%
{(Tensorproduktabbildung}) is continuous, and hence Fr\'echet differentiable.
\end{itemize}
\end{proposition}

In Corollary 4.4 in \cite{FALHACK} the following result, which is re-stated
here using the notations of the present paper, is proved as a consequence of
a similar result showed for tensors in Tucker format with bounded rank.

\begin{theorem}
\label{classical_best_approx} Let $\mathbf{V}_{D}=\left.
_{a}\bigotimes_{j\in D}V_{j}\right. $ and let $\{\mathbf{V}_{\alpha
_{j}\left. _{\Vert \cdot \Vert _{\alpha _{j}}}\right. }:2\leq j\leq d\}\cup
\{V_{j_{\Vert \cdot \Vert _{j}}}:1\leq j\leq d\}$ for $d\geq 3,$ be a
representation of a reflexive Banach tensor space $\mathbf{V}_{D_{\Vert
\cdot \Vert _{D}}}=\left. _{\Vert \cdot \Vert _{D}}\bigotimes_{j\in
D}V_{j}\right. ,$ in topological tree-based format such that

\begin{enumerate}
\item[(a)] $\|\cdot\|_D \gtrsim
\|\cdot\|_{\vee(V_{1_{\|\cdot\|_j}},\ldots,V_{d_{\|\cdot\|_d}})},$

\item[(b)] $\mathbf{V}_{\alpha_d} = V_{d-1} \otimes_a V_d,$ and $\mathbf{V}%
_{\alpha_{j}} = V_{j-1} \otimes_a \mathbf{V}_{\alpha_{j+1}},$ for $2 \le j
\le d-1,$ and

\item[(c)] $\|\cdot\|_{\alpha_j} :=
\|\cdot\|_{\vee(V_{\left.j-1\right._{\|\cdot\|_{j-1}}},\ldots,V_{d_{\|\cdot%
\|_d}})}$ for $2 \le j \le d.$
\end{enumerate}

Then for each $\mathbf{v} \in \mathbf{V}_{D_{\|\cdot\|_D}}$ there exists $%
\mathbf{u}_{best} \in \mathcal{FT}_{\le \mathfrak{r}}(\mathbf{V}_D)$ such
that 
\begin{align*}
\|\mathbf{v}-\mathbf{u}_{best}\|_D = \min_{\mathbf{u} \in \mathcal{FT}_{\le 
\mathfrak{r}}(\mathbf{V}_D)}\|\mathbf{v}-\mathbf{u}\|_D.
\end{align*}
\end{theorem}

It seems clear that tensor Banach spaces as we show in Example \ref{figZp1}
are not included in this framework. So a natural question is to ask if for a
representation in the topological tree-based format of a reflexive Banach
space the statement of Theorem \ref%
{classical_best_approx} is also true. To prove this, we will reformulate
some of the results given in \cite{FALHACK}. In the aforementioned paper,
the milestone to prove the existence of a best approximation is the
extension of the definition of minimal subspaces for tensors $\mathbf{v}\in 
\mathbf{V}_{D_{\Vert \cdot \Vert _{D}}}\setminus \mathbf{V}_{D}.$ To do this
the authors use a result similar  to the following lemma (see Lemma 3.8 in 
\cite{FALHACK}).

\begin{lemma}
\label{extension_Umin} Let $V_{j_{\|\cdot\|_j}}$ be a Banach space for $j\in
D,$ where $D$ is a finite index set, and $\alpha_1,\ldots,\alpha_m \subset
2^D \setminus \{D,\emptyset\},$ be such that $\alpha_i \cap \alpha_j =
\emptyset$ for all $i \neq j$ and $D = \bigcup_{i=1}^m \alpha_i.$ Assume
that if $\#\alpha_i \ge 2$ for some $1 \le i \le m,$ then $\mathbf{V}%
_{\left.\alpha_i\right._{\|\cdot\|_{\alpha_i}}}$ is a tensor Banach space.
Consider the tensor space 
\begin{equation*}
\mathbf{V}_D:= \left._a \bigotimes_{i=1}^m \mathbf{V}_{\left.\alpha_i%
\right._{\|\cdot\|_{\alpha_i}}} \right.
\end{equation*}
endowed with the injective norm $\|\cdot\|_{\vee(\mathbf{V}%
_{\left.\alpha_1\right._{\|\cdot\|_{\alpha_1}}},\ldots,\mathbf{V}%
_{\left.\alpha_m\right._{\|\cdot\|_{\alpha_m}}})}.$ Fix $1 \le k \le m,$
then given $\boldsymbol{\varphi}_{[\alpha_k]}\in \left._a \bigotimes_{i \neq
k} \mathbf{V}_{\left.\alpha_i\right._{\|\cdot\|_{\alpha_i}}}^*\right.$ the
map $id_{\alpha_k} \otimes \boldsymbol{\varphi}_{[\alpha_k]}$ belongs to $%
\mathcal{L}\left(\mathbf{V}_D, \mathbf{V}_{\left.\alpha_k\right._{\|\cdot%
\|_{\alpha_k}}} \right).$ Moreover, $\overline{id_{\alpha_k} \otimes 
\boldsymbol{\varphi}_{[\alpha_k]}} \in \mathcal{L}(\overline{\mathbf{V}_D}%
^{\|\cdot\|,},\mathbf{V}_{\left.\alpha_k\right._{\|\cdot\|_{\alpha_k}}})$
for any norm satisfying 
\begin{equation*}
\|\cdot\| \gtrsim \|\cdot\|_{\vee(\mathbf{V}_{\left.\alpha_1\right._{\|\cdot%
\|_{\alpha_1}}},\ldots,\mathbf{V}_{\left.\alpha_m\right._{\|\cdot\|_{%
\alpha_m}}})}.
\end{equation*}
\end{lemma}

Let $\{\mathbf{V}_{\alpha _{\Vert \cdot \Vert _{\alpha }}}\}_{\alpha \in
T_{D}\setminus \{D\}}$ be a representation of the Banach tensor space $%
\mathbf{V}_{D_{\Vert \cdot \Vert _{D}}}=\left. _{\Vert \cdot \Vert
_{D}}\bigotimes_{j\in D}V_{j}\right. ,$ in the topological tree-based format
and assume that the tensor product map $\bigotimes$ is $T_D$-continuous.
From Theorem \ref{ext_Banach}, we may assume that we have a tensor Banach
space 
\begin{equation*}
\mathbf{V}_{\alpha_{\|\cdot\|_{\alpha}}} = \left._{\|\cdot\|_{\alpha}}
\bigotimes_{\beta\in S(\alpha)} V_{\beta_{\|\cdot\|_{\beta}}} \right.
\end{equation*}
for each $\alpha \in T_D\setminus \mathcal{L}(T_D),$ and a Banach space $%
V_{j_{\|\cdot\|_j}}$ for $j \in \mathcal{L}(T_D).$ Let $\alpha \in
T_D\setminus \mathcal{L}(T_D).$ To simplify the notation we write for $A,B
\subset S(\alpha) $ 
\begin{equation*}
\|\cdot\|_{\vee(A)}:= \|\cdot\|_{\vee(\{\mathbf{V}_{\delta_{\|\cdot\|_{%
\delta}}}:\delta \in A\})},
\end{equation*}
and 
\begin{equation*}
\|\cdot\|_{\vee(A,\vee(B))}:= \|\cdot\|_{\vee(\{\mathbf{V}%
_{\delta_{\|\cdot\|_{\delta}}}:\delta \in A\}, \mathbf{X}_B)}
\end{equation*}
where 
\begin{equation*}
\mathbf{X}_B:= \left._{\|\cdot\|_{\vee(B)}} \bigotimes_{\beta \in B} \mathbf{%
V}_{\beta_{\|\cdot\|_{\beta}}}\right..
\end{equation*}
From Proposition~\ref{injective bounded below}(a), we can write 
\begin{align*}
\|\cdot\|_{\vee(S(\alpha))} = \|\cdot\|_{\vee(\beta,\vee(S(\alpha)\setminus
\beta))}  \label{tree_injective_norm}
\end{align*}
for each $\beta \in S(\alpha).$ From now on, we assume that 
\begin{align}
\|\cdot\|_{\alpha} \gtrsim \|\cdot\|_{\vee(S(\alpha))} \text{ for each }
\alpha \in T_D\setminus \mathcal{L}(T_D),
\end{align}
holds. Recall that Proposition~\ref{injective bounded below}(c) implies that
the tensor product map $\bigotimes$ is $T_D$-continuous. Since $%
\|\cdot\|_{\alpha} \gtrsim \|\cdot\|_{\vee(\beta,\vee(S(\alpha)\setminus
\beta))}$ holds for each $\beta\in S(\alpha)$, the tensor product map 
\begin{align*}
\bigotimes:(\mathbf{V}_{\beta_{\|\cdot\|_{\beta}}},\|\cdot\|_{\beta}) \times
\left( \left._{\|\cdot\|_{\vee(S(\alpha)\setminus \beta)}}
\bigotimes_{\delta \in S(\alpha)\setminus \{\beta\}}\mathbf{V}_{\delta_{\|\cdot%
\|_{\delta}}} \right. , \|\cdot\|_{\vee(S(\alpha)\setminus \beta)}\right)
\rightarrow (\mathbf{V}_{\alpha_{\|\cdot\|_{\alpha}}},\|\cdot\|_{\alpha})
\end{align*}
is also continuous for each $\beta \in S(\alpha).$ Moreover, by Theorem \ref%
{ext_Banach}, 
\begin{equation*}
\mathbf{V}_{\alpha _{\Vert \cdot \Vert _{\alpha }}}=\left. _{\Vert \cdot
\Vert _{\alpha }}\bigotimes_{\beta \in S(\alpha )}V_{\beta _{\Vert \cdot
\Vert _{\beta }}}\right. =\left. _{\Vert \cdot \Vert _{\alpha
}}\bigotimes_{\beta \in S(\alpha )}V_{\beta }\right. =\left. _{\Vert \cdot
\Vert _{\alpha }}\bigotimes_{j\in \alpha }V_{j}\right. ,
\end{equation*}%
holds for each $\alpha \in T_{D}\setminus \mathcal{L}(T_{D}).$ Observe, that 
$\mathbf{V}_{\alpha _{\Vert \cdot \Vert _{\alpha }}}^{\ast }\subset \mathbf{V%
}_{\alpha }^{\ast }$ for all $\alpha \in S(D).$ Take $\mathbf{V}_{D}=\left.
_{a}\bigotimes_{j\in D}V_{j}\right. .$ Since $\Vert \cdot \Vert _{D}\gtrsim
\Vert \cdot \Vert _{\vee (S(D ))},$ from Lemma \ref{extension_Umin} and
Proposition \ref{(Ualpha in Tensor Uj coro}, we can extend for $\mathbf{v}%
\in \mathbf{V}_{D_{\Vert \cdot \Vert _{D}}}\setminus \mathbf{V}_{D},$ the
definition of minimal subspace for each $\alpha \in S(D)$ as 
\begin{equation*}
U_{\alpha }^{\min }(\mathbf{v}):=\left\{ \overline{(id_{\alpha }\otimes 
\boldsymbol{\varphi }_{[\alpha ]})}(\mathbf{v}):\boldsymbol{\varphi }%
_{[\alpha ]}\in \left. _{a}\bigotimes_{\substack{ \beta \in S(D) \setminus
\{\alpha \}}}\mathbf{V}_{\beta }^{\ast }\right. \right\} .
\end{equation*}%
Observe that $\overline{(id_{\alpha }\otimes \boldsymbol{\varphi }_{[\alpha
]})}\in \mathcal{L}(\mathbf{V}_{D_{\Vert \cdot \Vert _{D}}},\mathbf{V}%
_{\alpha _{\Vert \cdot \Vert _{\alpha }}}).$ Recall that if $\mathbf{v}\in 
\mathbf{V}_{D}$ and $\alpha \notin \mathcal{L}(T_{D}),$ from Proposition \ref%
{inclusin_Umin}, we have $U_{\alpha }^{\min }(\mathbf{v})\subset \left.
_{a}\bigotimes_{\beta \in S(\alpha )}U_{\beta }^{\min }(\mathbf{v})\right.
\subset \left. _{a}\bigotimes_{\beta \in S(\alpha )}\mathbf{V}_{\beta
}\right. .$ Moreover, by Proposition \ref{(Ualpha in Tensor Uj coro}, for 
$\beta \in S(\alpha )$ we have 
\begin{align*}
U_{\beta }^{\min }(\mathbf{v})& = \mathrm{span}\, \left\{ (id_{\beta
}\otimes \boldsymbol{\varphi }_{[\beta ]})(\mathbf{v}_{\alpha }): \mathbf{v}%
_{\alpha}\in U_{\alpha }^{\min }(\mathbf{v}) \text{ and } \boldsymbol{%
\varphi }_{[\beta ]}\in \left._{a}\bigotimes_{\substack{ \delta\in
S(\alpha)\setminus \{\beta\}}} \mathbf{V}_{\delta }^{\ast } \right. \right\}
\\
& =\mathrm{span}\, \left\{ (id_{\beta }\otimes \boldsymbol{\varphi }%
_{[\beta]})\circ (id_{\alpha }\otimes \boldsymbol{\varphi }_{[\alpha ]})(%
\mathbf{v}):\boldsymbol{\varphi }_{[\alpha ]}\in \left. _{a}\bigotimes 
_{\substack{ \mu \in S(D) \setminus \{\alpha\} }}\mathbf{V}_{\mu }^{\ast
}\right. \text{ and }\boldsymbol{\varphi }_{[\beta ]}\in \left.
_{a}\bigotimes_{\substack{ \delta\in S(\alpha) \setminus \{\beta\}}} \mathbf{%
V}_{\delta }^{\ast }\right. \right\} .
\end{align*}%
Thus, $(id_{\alpha }\otimes \boldsymbol{\varphi }_{[\alpha ]})(\mathbf{v}%
)\in U_{\alpha }^{\min }(\mathbf{v})\subset \mathbf{V}_{\alpha }\subset 
\mathbf{V}_{\alpha _{\Vert \cdot \Vert _{\alpha }}},$ and hence 
\begin{equation*}
(id_{\beta }\otimes \boldsymbol{\varphi }_{[\beta ]})\circ (id_{\alpha
}\otimes \boldsymbol{\varphi }_{[\alpha ]})(\mathbf{v})\in U_{\beta }^{\min
}(\mathbf{v})\subset \mathbf{V}_{\beta }\subset \mathbf{V}_{\beta _{\Vert
\cdot \Vert _{\beta }}},
\end{equation*}%
when $\#\beta \geq 2.$ However, if $\mathbf{v}\in \mathbf{V}_{D_{\Vert \cdot
\Vert _{D}}}\setminus \mathbf{V}_{D}$ then $\overline{(id_{\alpha }\otimes 
\boldsymbol{\varphi }_{[\alpha ]})}(\mathbf{v})\in U_{\alpha }^{\min }(%
\mathbf{v})\subset \mathbf{V}_{\alpha _{\Vert \cdot \Vert _{\alpha }}}.$
Since $\Vert \cdot \Vert _{\alpha }\gtrsim \Vert \cdot \Vert _{\vee
(S(\alpha ))}$ also by Lemma \ref{extension_Umin} we have $\overline{%
id_{\beta }\otimes \boldsymbol{\varphi }_{[\beta ]}}\in \mathcal{L}(\mathbf{V%
}_{\alpha _{\Vert \cdot \Vert _{\alpha }}},\mathbf{V}_{\beta _{\Vert \cdot
\Vert _{\beta }}}).$ In consequence, a natural extension of the definition
of minimal subspace $U_{\beta }^{\min }(\mathbf{v}),$ for $\mathbf{v}\in 
\mathbf{V}_{D_{\Vert \cdot \Vert _{D}}}\setminus \mathbf{V}_{D},$ is given
by 
\begin{equation*}
U_{\beta }^{\min }(\mathbf{v}):=\mathrm{span}\, \left\{ \overline{%
(id_{\beta}\otimes \boldsymbol{\varphi }_{[\beta ]})}\circ \overline{%
(id_{\alpha }\otimes \boldsymbol{\varphi }_{[\alpha ]})}(\mathbf{v}):%
\boldsymbol{\varphi }_{[\alpha ]}\in \left. _{a}\bigotimes_{\substack{ \mu
\in S(D) \setminus \{\alpha\}}}\mathbf{V}_{\mu }^{\ast }\right. \text{ and } 
\boldsymbol{\varphi }_{[\beta ]}\in \left._{a}\bigotimes_{\substack{ %
\delta\in S(\alpha) \setminus \{\beta\}}} \mathbf{V}_{\delta }^{\ast
}\right. \right\}.
\end{equation*}%
To simplify the notation, we can write 
\begin{equation*}
\overline{(id_{\beta }\otimes \boldsymbol{\varphi }_{[\beta ,\alpha ]})}(%
\mathbf{v}):=\overline{(id_{\beta }\otimes \boldsymbol{\varphi }_{[\beta ]})}%
\circ \overline{(id_{\alpha }\otimes \boldsymbol{\varphi }_{[\alpha ]})}(%
\mathbf{v})
\end{equation*}%
where $\boldsymbol{\varphi }_{[\beta ,\alpha ]}:=\boldsymbol{\varphi }%
_{[\alpha ]}\otimes \boldsymbol{\varphi }_{[\beta ]}\in \left( \left.
_{a}\bigotimes_{\substack{ \mu \in S(D) \setminus \{\alpha\} }}\mathbf{V}%
_{\mu }^{\ast }\right. \right) \otimes _{a}\left( \left. _{a}\bigotimes 
_{\substack{ \delta \in S(\alpha ) \setminus \{\beta\} }}\mathbf{V}_{\delta
}^{\ast }\right. \right) $ and $\overline{(id_{\beta }\otimes \boldsymbol{%
\varphi }_{[\beta ,\alpha ]})}\in \mathcal{L}(\mathbf{V}_{D_{\Vert \cdot
\Vert _{D}}},\mathbf{V}_{\beta _{\Vert \cdot \Vert _{\beta }}}).$ Proceeding
inductively, from the root to the leaves, we define the minimal subspace $%
U_{j}^{\min }(\mathbf{v})$ for each $j\in \mathcal{L}(T_{D})$ such that
there exists $\eta \in T_{D}\setminus \{D\}$ with $j\in S(\eta )$ as 
\begin{equation*}
U_{j}^{\min }(\mathbf{v}):=\mathrm{span}\,\left\{ \overline{(id_{j}\otimes 
\boldsymbol{\varphi }_{[j,\eta ,\ldots ,\beta ,\alpha ]})}(\mathbf{v}):%
\boldsymbol{\varphi }_{[j,\eta,\ldots ,\beta ,\alpha ]}\in \mathbf{W}%
_{j}\right\} ,
\end{equation*}%
where 
\begin{equation*}
\mathbf{W}_{j}:=\left( \left. _{a}\bigotimes_{\substack{ \mu \in S(D)
\setminus \{\alpha\} }} \mathbf{V}_{\mu }^{\ast }\right. \right) \otimes
_{a}\left( \left. _{a}\bigotimes_{\substack{ \delta \in S(\alpha ) \setminus
\{\beta\} }}\mathbf{V}_{\delta }^{\ast }\right. \right) \otimes _{a}\cdots
\otimes _{a}\left( \left. _{a}\bigotimes_{\substack{ k\in S(\eta )\setminus
\{j\} }}V_{k}^{\ast }\right. \right) .
\end{equation*}%
With this extension the following result  can be shown (see Lemma 3.13 in 
\cite{FALHACK}).

\begin{lemma}
Let $\{\mathbf{V}_{\alpha _{\Vert \cdot \Vert _{\alpha }}}\}_{\alpha \in
T_{D}\setminus \{D\}}$ be a representation of the Banach tensor space $%
\mathbf{V}_{D_{\Vert \cdot \Vert _{D}}}=\left. _{\Vert \cdot \Vert
_{D}}\bigotimes_{j\in D}V_{j}\right. ,$ in the topological tree-based format
and assume that \eqref{tree_injective_norm} holds. Let $\{\mathbf{v}%
_{n}\}_{n\geq 0}\subset \mathbf{V}_{D_{\Vert \cdot \Vert _{D}}}$ with $%
\mathbf{v}_{n}\rightharpoonup \mathbf{v},$ and $\mu \in T_{D}\setminus
(\{D\}\cup \mathcal{L}(T_{D}))$. Then for each $\gamma \in S(\mu )$ we have 
\begin{equation*}
\overline{(id_{\gamma }\otimes \boldsymbol{\varphi }_{[\gamma ,\mu ,\cdots
,\beta ,\alpha ]})}(\mathbf{v}_{n})\rightharpoonup \overline{(id_{\gamma
}\otimes \boldsymbol{\varphi }_{[\gamma ,\mu ,\cdots ,\beta ,\alpha ]})}(%
\mathbf{v})\text{ in }\mathbf{V}_{\gamma _{\Vert \cdot \Vert _{\gamma }}},
\end{equation*}%
for all $\boldsymbol{\varphi }_{[\gamma ,\mu ,\cdots ,\beta ,\alpha ]}\in
\left( \left. _{a}\bigotimes_{\substack{ \mu \in S(D)\setminus \{\alpha \}}}%
\mathbf{V}_{\mu }^{\ast }\right. \right) \otimes _{a}\left( \left.
_{a}\bigotimes_{\substack{ \delta \in S(\alpha )\setminus \{\beta \}}}%
\mathbf{V}_{\delta }^{\ast }\right. \right) \otimes _{a}\cdots \otimes
_{a}\left( \left. _{a}\bigotimes_{\substack{ \eta \in S(\mu )\setminus
\{\gamma \}}}V_{\eta }^{\ast }\right. \right) .$
\end{lemma}

Then in a similar way as Theorem 3.15 in \cite{FALHACK} the following
theorem can be shown.

\begin{theorem}
Let $\{\mathbf{V}_{\alpha _{\Vert \cdot \Vert _{\alpha }}}\}_{\alpha \in
T_{D}\setminus \{D\}}$ be a representation of the Banach tensor space $%
\mathbf{V}_{D_{\Vert \cdot \Vert _{D}}}=\left. _{\Vert \cdot \Vert
_{D}}\bigotimes_{j\in D}V_{j}\right. ,$ in the topological tree-based format
and assume that \eqref{tree_injective_norm} holds. Let $\{\mathbf{v}%
_{n}\}_{n\geq 0}\subset \mathbf{V}_{D_{\Vert \cdot \Vert _{D}}}$ with $%
\mathbf{v}_{n}\rightharpoonup \mathbf{v},$ then 
\begin{equation*}
\dim \overline{U_{\alpha }^{\min }(\mathbf{v})}^{\Vert \cdot \Vert _{\alpha
}}=\dim U_{\alpha }^{\min }(\mathbf{v})\leq \liminf_{n\rightarrow \infty
}\dim U_{\alpha }^{\min }(\mathbf{v}_{n}),
\end{equation*}%
for all $\alpha \in T_{D}\setminus \{D\}.$
\end{theorem}

Now, following the proof of Theorem 4.1 in \cite{FALHACK} we obtain the
final theorem.

\begin{theorem}
\label{best_approximation_FT} Let $\mathbf{V}_{D}=\left.
_{a}\bigotimes_{j\in D}V_{j}\right. $ and let $\{\mathbf{V}_{\alpha _{\Vert
\cdot \Vert _{\alpha }}}\}_{\alpha \in T_{D}\setminus \{D\}}$ be a
representation of a reflexive Banach tensor space $\mathbf{V}_{D_{\Vert
\cdot \Vert _{D}}}=\left. _{\Vert \cdot \Vert _{D}}\bigotimes_{j\in
D}V_{j}\right. $ in the topological tree-based format and assume that \eqref%
{tree_injective_norm} holds. Then the set $\mathcal{FT}_{\leq \mathfrak{r}}(%
\mathbf{V}_{D})$ is weakly closed in $\mathbf{V}_{D_{\Vert \cdot \Vert _{D}}}
$ and hence for each $\mathbf{v}\in \mathbf{V}_{D_{\Vert \cdot \Vert _{D}}}$
there exists $\mathbf{u}_{best}\in \mathcal{FT}_{\leq \mathfrak{r}}(\mathbf{V%
}_{D})$ such that 
\begin{equation*}
\Vert \mathbf{v}-\mathbf{u}_{best}\Vert _{D}=\min_{\mathbf{u}\in \mathcal{FT}%
_{\leq \mathfrak{r}}(\mathbf{V}_{D})}\Vert \mathbf{v}-\mathbf{u}\Vert _{D}.
\end{equation*}
\end{theorem}




\end{document}